\documentclass[12pt]{amsart}

\usepackage{amsmath, amssymb, amsthm}
\usepackage{hyperref}
\usepackage{microtype}

\newtheorem{theorem}{Theorem}[section]
\newtheorem{lemma}[theorem]{Lemma}
\newtheorem{corollary}[theorem]{Corollary}
\newtheorem{proposition}[theorem]{Proposition}
\theoremstyle{definition}
\newtheorem{definition}[theorem]{Definition}
\newtheorem{remark}[theorem]{Remark}

\newcommand{\D}{\mathbb{D}}
\newcommand{\Z}{\mathbb{Z}}
\newcommand{\Zi}{\mathbb{Z}[i]}
\newcommand{\C}{\mathbb{C}}
\newcommand{\R}{\mathbb{R}}
\newcommand{\abs}[1]{\left|#1\right|}
\newcommand{\OD}{\mathcal{O}(\D)}
\newcommand{\mR}{\mathcal{R}}
\newcommand{\mRR}{\mathcal{R}_{\mathbb{R}}}

\title[Prime Spectrum of Power Series Rings]{The Prime Spectrum of Rings of\\Integer-Coefficient Power Series on the Disk}
\author{Jon Bannon}\author{David Feldman}
\date{}

\begin{document}

\maketitle

\begin{abstract}
This is a sequel to \cite{PaperI}, where it is shown that a discrete
effective divisor on the open unit disk $\D$ is the zero divisor of a
holomorphic function with integer Taylor coefficients if and only if it
is invariant under complex conjugation.  Building on that realization
theorem and on the unit characterization established there, we study the
prime and maximal spectra of the rings $\mR = \Zi[[z]] \cap \OD$ and
$\mRR = \Z[[z]] \cap \OD$.  We give a trichotomy for maximal ideals,
classify the prime ideals lying outside the class $\mathfrak{P}_1$ of
primes containing a unit-constant-term element, attach to each element
of $\mathfrak{P}_1$ a complete analytic invariant, and construct an
injection from ultrafilters on admissible divisors into $\mathfrak{P}_1$
by an ultraproduct device.  We prove the point-evaluation ideals $P_a$
are prime for all $a \in \D$ and maximal for real $a$, and deduce a
conditional B\'ezout property for disjoint divisors.  All results
recalled from \cite{PaperI} are used only as cited; no proof of that
paper is reproduced here.
\end{abstract}

\section{Introduction}
\label{sec:intro}

The companion paper \cite{PaperI} proves that the only arithmetic
obstruction to prescribing the zero divisor of an integer-coefficient
holomorphic function on the unit disk $\D = \{z \in \C : \abs{z} < 1\}$
is conjugation invariance, and derives a first layer of ring-theoretic
consequences for the coefficient-restricted algebras
\[
  \mR = \Zi[[z]] \cap \OD, \qquad \mRR = \Z[[z]] \cap \OD,
\]
where $\Zi = \Z[i]$.  Those foundational results---the realization
theorems, the characterization of units, the factorization of
$\OD$-elements up to units, and the injectivity of the contraction
$\mathrm{MaxSpec}(\OD) \to \mathrm{Spec}(\mR)$---have been formally
verified.  We take them as given (Section~\ref{sec:recall}) and develop
from them a systematic description of the prime and maximal spectra.

The present paper contains no overlap with \cite{PaperI}: the
construction of realizing functions, its convergence analysis, and the
elementary-factor machinery live entirely in the companion paper and are
invoked here only through the statements collected in
Section~\ref{sec:recall}.  Everything from Section~\ref{sec:trichotomy}
onward is new.

\medskip\noindent\textbf{Results.}
After recalling the needed facts, we establish a trichotomy for maximal
ideals of $\mRR$ and show one of the three types cannot occur
(Section~\ref{sec:trichotomy}).  We then introduce the three ambient
classes of prime ideals $\mathfrak{P} \supseteq \mathfrak{P}_1 \supseteq
\mathfrak{M} \cap \mathfrak{P}_1$ and classify the primes outside
$\mathfrak{P}_1$ completely: they are the ideals $(p)$ (prime,
non-maximal) and $(z,p)$ (maximal), $p$ a rational prime
(Section~\ref{sec:outside}).  For primes \emph{inside} $\mathfrak{P}_1$
we attach an analytic invariant $\mathcal{Z}(P)$---the family of zero
divisors of the unit-constant-term elements of $P$---and prove it is
complete (Section~\ref{sec:invariant}).  The invariant satisfies a
partition property, and we realize every ultrafilter on an admissible
divisor as the invariant of a prime ideal built by an ultraproduct
construction (Section~\ref{sec:ultra}).  Finally we treat the
point-evaluation ideals $P_a = \ker(\mathrm{ev}_a)$, proving primeness
for all $a \in \D$ and maximality for real $a$ (and for all $a$ in the
Gaussian ring $\mR$), and deduce a conditional B\'ezout theorem for
disjoint divisors (Sections~\ref{sec:pointeval}--\ref{sec:bezout}).

Throughout, the arguments are stated for $\mRR$ for concreteness; they
apply verbatim to $\mR$ with $\Zi$ in place of $\Z$ unless noted.  We
write $\mathfrak{n}_0 = \{g : g(0) = 0\}$ for the augmentation ideal, and
$[z^n]f$ for the coefficient of $z^n$ in $f$.

\section{Notation and results recalled from \cite{PaperI}}
\label{sec:recall}

We recall the statements used below.  Proofs are in \cite{PaperI}; the
locators are placeholders to be synchronized with the final numbering of
the companion paper.

\begin{theorem}[Integer realization; {\cite[Thm.~1.1]{PaperI}}]
\label{thm:main}
An effective divisor $D$ on $\D$ is the zero divisor of a holomorphic
function on $\D$ with Taylor coefficients in $\Z$ if and only if $D$ is
invariant under complex conjugation.
\end{theorem}

\begin{theorem}[Gaussian-integer realization; {\cite[Thm.~1.2]{PaperI}}]
\label{thm:Zi}
Every effective divisor on $\D$ is the zero divisor of a holomorphic
function on $\D$ with Taylor coefficients in $\Zi$.
\end{theorem}

In particular (Theorem~\ref{thm:main}), every admissible divisor in $\D$
is the zero set, with multiplicity, of some element of $\mRR$ with
constant term $1$; this is the only input to the partition and
ultraproduct arguments below.

\begin{proposition}[Units; {\cite[Prop.~4.1]{PaperI}}]
\label{prop:units}
An element $f \in \mR$ is a unit if and only if $f(0) \in \Zi^\times =
\{\pm 1, \pm i\}$ and $f$ is nowhere vanishing on $\D$.  The analogous
statement holds for $\mRR$, with $\Zi^\times$ replaced by $\Z^\times =
\{\pm 1\}$.
\end{proposition}

We also use, from \cite{PaperI}, that every $f \in \OD$ factors as $f =
g\,u$ with $g \in \mR$ and $u \in \OD^\times$, that consequently the
contraction map
\[
  \phi \colon \mathrm{MaxSpec}(\OD) \to \mathrm{Spec}(\mR), \qquad
  \mathfrak{m} \mapsto \mathfrak{m} \cap \mR,
\]
is injective, and that $\mR/\mathfrak{n}_0 \cong \Zi$ (respectively
$\mRR/\mathfrak{n}_0 \cong \Z$).  These frame the picture but are not
needed in the proofs that follow.

\section{A trichotomy for maximal ideals}
\label{sec:trichotomy}

\begin{proposition}[Trichotomy of maximal ideals]
\label{prop:trichotomy}
Every maximal ideal $M \subset \mRR$ falls into exactly one of the
following types:
\begin{enumerate}
\item[(i)] $z \in M$,
\item[(ii)] $n \in M$ for some integer $n$ with $\abs{n} > 1$, but
  $z \notin M$,
\item[(iii)] neither $z$ nor any integer constant belongs to $M$.
\end{enumerate}
Moreover, type~\emph{(ii)} cannot occur: no maximal ideal of $\mRR$
contains an integer but not $z$.
\end{proposition}

\begin{proof}
The three types are mutually exclusive and exhaustive by definition.
It remains to show type~(ii) cannot occur.  Suppose $n \in M$ and $z
\notin M$.  For any $f \in M$, since $f - f(0) \in M$ and $z \notin M$
and $M$ is prime, $(f - f(0))/z \in M$.  Iterating, every coefficient
of $f$ is divisible by $n$, so $M \subseteq n\mRR$.  But $n \in M$
gives $n\mRR \subseteq M$, hence $M = n\mRR$.  For $n = p$ prime,
$\mRR/p\mRR \cong \mathbb{F}_p[[z]]$, which has proper ideal $(z)$,
contradicting maximality of $M = p\mRR$.  For composite $n$, a prime
factor $p \mid n$ gives $p \in M$ and the same argument applies.
\end{proof}

\begin{definition}
A maximal ideal of type~(iii) is called a \emph{type~(iii) maximal
ideal}.
\end{definition}

We fix, once and for all, three classes of prime ideals of $\mRR$,
using fraktur notation:
\begin{itemize}
\item $\mathfrak{P} = \mathrm{Spec}(\mRR)$: all prime ideals of $\mRR$.
\item $\mathfrak{P}_1$: prime ideals $P$ containing at least one element
  of the form $1 + zF$ (i.e.\ with unit constant term).
\item $\mathfrak{M}$: maximal ideals of $\mRR$.
\end{itemize}
These satisfy $\mathfrak{M} \cap \mathfrak{P}_1 \subseteq \mathfrak{P}_1
\subseteq \mathfrak{P}$.  Every type~(iii) maximal ideal lies in
$\mathfrak{P}_1$ by Proposition~\ref{prop:interesting}(i) below;
type~(i) maximal ideals (containing $z$) do not contain
unit-constant-term elements, since if $z \in P$ and $1 + zF \in P$ then
$1 = (1 + zF) - z \cdot F \in P$, contradicting $P$ proper.  Whether
every element of $\mathfrak{P}_1$ is maximal is open
(Section~\ref{sec:open}).

\section{Prime ideals outside $\mathfrak{P}_1$}
\label{sec:outside}

\begin{proposition}[Prime ideals outside $\mathfrak{P}_1$]
\label{prop:primes_outside}
Every prime ideal $P \in \mathfrak{P} \setminus \mathfrak{P}_1$ is of
exactly one of the following forms:
\begin{enumerate}
\item[(a)] $P = (p)$ for a rational prime $p$, i.e.\ the ideal generated
  by the constant series $p$.  These are prime but not maximal.
\item[(b)] $P = (z, p)$ for a rational prime $p$.  These are maximal.
\end{enumerate}
\end{proposition}

\begin{proof}
Let $P \in \mathfrak{P} \setminus \mathfrak{P}_1$.  Since $P \notin
\mathfrak{P}_1$, no element of $P$ has constant term $\pm 1$, so the set
of constant terms $\{f(0) : f \in P\}$ is an ideal $(d) \subseteq \Z$
with $d \geq 2$.

\medskip\noindent\emph{Case (b): $z \in P$.}  The image of $P$ in
$\mRR/(z) \cong \Z$ is a prime ideal of $\Z$, hence $(p)$ for some prime
$p$.  Thus $P = (z, p)$.  Since $\mRR/(z,p) \cong \Z/(p) = \mathbb{F}_p$
is a field, $P$ is maximal.

\medskip\noindent\emph{Case (a): $z \notin P$.}  Since $P$ is prime and
$z \notin P$: for any $f \in P$ with $f = z^k g$ and $g(0) \neq 0$,
primeness gives $g \in P$.  Thus $P$ is closed under removing leading
powers of $z$.

We claim every coefficient of every element of $P$ is divisible by $p$.
Consider the quotient $\mRR/(p)$, whose elements are power series $\sum
\bar a_n z^n$ with $\bar a_n \in \Z/(p) = \mathbb{F}_p$; since $p$ is a
nonzero integer and hence invertible in $\C$, convergence on $\D$ is
automatic.  Let $\pi \colon \mRR \to \mRR/(p)$ be the quotient map, and
$\bar P = \pi(P)$.  Since $\pi$ is surjective and $P$ is prime, $\bar P$
is a prime ideal of $\mRR/(p)$.

Every element of $\bar P$ has zero constant term: if $f \in P$ then
$f(0) \in (p)\Z$, so $\pi(f)(0) = \overline{f(0)} = 0$.  Since $z
\notin P$, $\bar z = \pi(z) \notin \bar P$.  For any $\bar f \in \bar P$
with $\bar f(0) = 0$, write $\bar f = \bar z \cdot \bar g$; then $\bar z
\notin \bar P$ and $\bar P$ prime give $\bar g \in \bar P$, so $\bar
g(0) = \bar a_1 = 0$.  Iterating, every coefficient $\bar a_n = 0$ in
$\Z/(p)$, i.e.\ $p \mid a_n$ for all $n$.  Hence $P \subseteq p\mRR$.

Since $(d) = (p)$, there exists $f_0 \in P$ with $f_0(0) = p$.  As all
coefficients of $f_0$ are divisible by $p$, write $f_0 = p \cdot g_0$
with $g_0(0) = 1$.  From $f_0 = p g_0 \in P$ and primeness, either $p
\in P$ or $g_0 \in P$; the latter forces $g_0(0) = 1 \in (p)\Z$,
impossible.  Hence $p \in P$, so $p\mRR \subseteq P$ and $P = p\mRR =
(p)$.  Finally $\mRR/(p) \cong \mathbb{F}_p[[z]]$ (the coefficientwise
reduction is a surjection with kernel $(p)$; surjectivity holds because
any $\sum \bar a_n z^n$ lifts to $\sum a_n z^n$ with $a_n \in
\{0,\ldots,p-1\}$, of radius of convergence $\geq 1$), which is not a
field, so $P = (p)$ is prime but not maximal.
\end{proof}

\begin{remark}
Proposition~\ref{prop:primes_outside} yields
\[
  \mathfrak{P} \setminus \mathfrak{P}_1
  = \{(p) : p \text{ prime}\} \cup \{(z,p) : p \text{ prime}\}.
\]
The ideals $(z,p)$ are the type~(i) maximal ideals.  Each $(p)$ is a
non-maximal prime, properly contained in $(z,p)$ but in no type~(iii)
maximal ideal (a type~(iii) maximal ideal contains no integer constant,
yet $(p)$ contains $p$).
\end{remark}

\section{Structure of primes in $\mathfrak{P}_1$}
\label{sec:invariant}

We first record a consequence of the unit characterization from
\cite{PaperI}.

\begin{lemma}[Unit quotient lemma]
\label{lem:unitquotient}
If $f_1, f_2 \in \mRR$ both have constant term $1$ and the same zero set
$Z(f_1) = Z(f_2)$ with multiplicities, then $f_1/f_2 \in \mRR^\times$.
\end{lemma}

\begin{proof}
Write $f_2 = \sum_{n \geq 0} a_n z^n$ with $a_0 = 1$ and $a_n \in \Z$.
The ratio $h = f_1/f_2$ extends holomorphically to $\D$ (the zeros of
$f_2$ are exactly those of $f_1$ with equal multiplicities, so $h$ has
no poles), is nowhere vanishing (identical zero sets), and $h(0) = 1$.
Writing $h = \sum_{n \geq 0} b_n z^n$ and comparing coefficients in $f_1
= f_2 \cdot h$ gives $b_0 = 1$ and
\[
  b_n = [z^n]f_1 - \sum_{k=1}^n a_k b_{n-k}, \quad n \geq 1.
\]
Since $[z^n]f_1 \in \Z$, $a_k \in \Z$, and $b_0 \in \Z$, induction gives
$b_n \in \Z$ for all $n$.  Because $h$ is nowhere vanishing on $\D$, $1/h
\in \OD$, so $h \in \mRR$.  Being nowhere vanishing with $h(0) = 1 \in
\Z^\times$, $h$ is a unit of $\mRR$ by Proposition~\ref{prop:units}.
\end{proof}

\begin{proposition}[Structure of type~(iii) maximal ideals]
\label{prop:interesting}
Let $M$ be a type~(iii) maximal ideal of $\mRR$.  Then:
\begin{enumerate}
\item[(i)] $M$ contains an element of the form $1 + zF$.  Indeed the set
  $I = \{g(0) : g \in M\} \subset \Z$ is an ideal $(d)$, $d \geq 0$; the
  cases $d = 0$ and $d > 1$ both contradict the type~(iii) hypothesis,
  so $d = 1$.
\item[(ii)] An element $f = 1 + zF \in \mRR$ is a unit of $\mRR$ if and
  only if it is nowhere vanishing on $\D$; this is the $\mRR$ case of
  Proposition~\ref{prop:units}, the constant term $1 \in \Z^\times$
  being automatic.  In particular, elements of $M$ of the form $1 + zF$
  are non-units, hence vanish somewhere in $\D$.
\item[(iii)] The elements of the form $1 + zF$ lying in $M$ generate $M$
  as an ideal.
\item[(iv)] Two elements of $\mRR$ of the form $1 + zF$ with the same
  zero set differ by a unit of $\mRR$ (Lemma~\ref{lem:unitquotient}).
\end{enumerate}
\end{proposition}

\begin{proof}
(i) The set $I = \{g(0) : g \in M\}$ is an ideal of $\Z$ (closed under
addition since $M$ is an additive subgroup, and under integer
multiplication since $\Z \subset \mRR$ and $M$ is an ideal), so $I = (d)$
for some $d \geq 0$.

If $d = 0$: every element of $M$ has zero constant term.  Let $f \in M$
be nonzero; it vanishes to some finite order $k \geq 1$ at $0$, so $f =
z^k u$ with $u(0) \neq 0$.  Since $f = z \cdot (z^{k-1}u) \in M$ and $M$
is prime, either $z \in M$ or $z^{k-1}u \in M$; iterating, either $z \in
M$ or $u \in M$.  But $u \in M$ gives $u(0) \in I = (0)$, contradicting
$u(0) \neq 0$.  Hence $z \in M$, contradicting type~(iii).

If $d > 1$: evaluation at $0$ induces an injection $\mRR/M
\hookrightarrow \Z/(d)$.  As $M$ is maximal, $\mRR/M$ is a field, and a
field embeds in $\Z/(d)$ only if $d = p$ is prime.  Then $K =
\ker(\mRR \xrightarrow{f \mapsto f(0) \bmod p} \Z/(p))$ is a proper ideal
containing both $M$ and the constant series $p$; maximality of $M$ gives
$K = M$, so $p \in M$, contradicting type~(iii).

Hence $d = 1$, and some element of $M$ has constant term $\pm 1$;
negating if needed gives constant term $1$.

(ii) This is Proposition~\ref{prop:units} for $\mRR$ applied to $f = 1 +
zF$ (whose constant term is the unit $1$): $f$ is a unit iff it is
nowhere vanishing.  Elements of $M$ are non-units, so those of the form
$1 + zF$ must vanish somewhere in $\D$.

(iii) By part~(i) every element of $M$ has the form $z^k h$ with $h \in
M$ of constant term $\pm 1$ (remove the leading power of $z$; $z \notin
M$ and primeness keep the cofactor in $M$).  Since $z^k \in \mRR$ and $h$
(or $-h$) has constant term $1$, the unit-constant-term elements of $M$
generate $M$.

(iv) Immediate from Lemma~\ref{lem:unitquotient}.
\end{proof}

\begin{definition}
For a type~(iii) maximal ideal $M$, define its \emph{analytic invariant}
\[
  \mathcal{Z}(M) = \{ Z(f) : f \in M,\; f = 1 + zF \text{ for some }
  F \in \mRR \},
\]
where $Z(f) \subset \D$ is the zero set of $f$ counted with
multiplicities.
\end{definition}

\begin{corollary}[$\mathcal{Z}$ is complete on type~(iii) maximals]
\label{cor:Zinvariant}
For type~(iii) maximal ideals $M_1, M_2 \subset \mRR$,
\[
  M_1 = M_2 \iff \mathcal{Z}(M_1) = \mathcal{Z}(M_2).
\]
\end{corollary}

\begin{proof}
Only ($\Leftarrow$) needs proof.  By
Proposition~\ref{prop:interesting}(iii), $M$ is generated by its
unit-constant-term set $S_M = \{f \in M : f = 1 + zF\}$.  If
$\mathcal{Z}(M_1) = \mathcal{Z}(M_2)$ then for each $f \in S_{M_1}$ there
is $g \in S_{M_2}$ with $Z(g) = Z(f)$; by
Proposition~\ref{prop:interesting}(iv), $f = ug$ for a unit $u$, so $f
\in M_2$.  Hence $S_{M_1} \subseteq M_2$ and $M_1 \subseteq M_2$; by
maximality $M_1 = M_2$.  The reverse inclusion is symmetric.
\end{proof}

\begin{remark}
\label{rem:Zunion}
The family $\mathcal{Z}(M)$ is closed under finite unions: if $f, g \in
M$ then $fg \in M$ and $Z(fg) = Z(f) \cup Z(g)$.  By
Theorem~\ref{thm:main}, every admissible divisor in $\D$ is $Z(h)$ for
some $h \in \mRR$ with constant term $1$, so $\mathcal{Z}(M)$ takes
values in families of admissible divisors.
\end{remark}

The invariant extends to all of $\mathfrak{P}_1$.  For $P \in
\mathfrak{P}_1$, define $\mathcal{Z}(P)$ by the same formula.

\begin{proposition}[$\mathcal{Z}$ is complete on $\mathfrak{P}_1$]
\label{prop:prime1invariant}
For $P_1, P_2 \in \mathfrak{P}_1$: $P_1 = P_2 \iff \mathcal{Z}(P_1) =
\mathcal{Z}(P_2)$.
\end{proposition}

\begin{proof}
The forward direction is trivial.  For the reverse, first note that
unit-constant-term elements generate any $P \in \mathfrak{P}_1$.  Given
$f \in P$, write $f = z^k(a_0 + zF)$; since $P$ contains some $1 + zG$,
we have $z \notin P$ (else $1 = (1 + zG) - z \cdot G \in P$), so
primeness gives $a_0 + zF \in P$.  The constant terms of elements of $P$
form an ideal of $\Z$ containing $1$, hence all of $\Z$, so every element
of $P$ is an $\mRR$-combination of unit-constant-term elements of $P$.

Now if $\mathcal{Z}(P_1) = \mathcal{Z}(P_2)$, for each unit-constant-term
$f \in P_1$ there is a unit-constant-term $g \in P_2$ with $Z(g) =
Z(f)$; by Proposition~\ref{prop:interesting}(iv), $f = ug$ for a unit
$u$, so $f \in P_2$.  Hence $P_1 \subseteq P_2$, and by symmetry $P_1 =
P_2$.
\end{proof}

\begin{proposition}[Partition property of $\mathcal{F}_Z(P)$]
\label{prop:prime1ultrafilter}
Let $P \in \mathfrak{P}_1$ and $f = 1 + zF \in P$ with $Z = Z(f)$.  Put
\[
  \mathcal{F}_Z(P) = \{ Z(h) : h \in P,\; h = 1 + zH,\; Z(h) \subseteq Z \}.
\]
For every partition $Z = Z_1 \sqcup Z_2$, at least one of $Z_1, Z_2$ lies
in $\mathcal{F}_Z(P)$.
\end{proposition}

\begin{proof}
By Theorem~\ref{thm:main}, obtain $f_1, f_2 \in \mRR$ with constant term
$1$ and $Z(f_i) = Z_i$.  Then $f_1 f_2$ has constant term $1$ and $Z(f_1
f_2) = Z_1 \cup Z_2 = Z = Z(f)$, so by
Proposition~\ref{prop:interesting}(iv), $f_1 f_2 = u f$ for a unit $u$.
As $f \in P$, $f_1 f_2 \in P$; primeness gives $f_1 \in P$ or $f_2 \in
P$, i.e.\ $Z_1 \in \mathcal{F}_Z(P)$ or $Z_2 \in \mathcal{F}_Z(P)$.
\end{proof}

\begin{remark}
The partition property is all that follows from primeness alone.  The
full ultrafilter property---that \emph{exactly} one of $Z_1, Z_2$ lies in
$\mathcal{F}_Z(P)$---requires that $P$ not contain elements with disjoint
zero sets, which depends on maximality or further structural information
about $P$.
\end{remark}

\section{Ultrafilters give prime ideals via ultraproducts}
\label{sec:ultra}

\begin{proposition}[Ultraproduct construction]
\label{prop:ultraproductprime}
Let $Z$ be an admissible divisor in $\D$ and $\mathcal{U}$ an ultrafilter
on the underlying point set $|Z|$ (ignoring multiplicities).  For each $W
\in \mathcal{U}$ let $f_W \in \mRR$ have constant term $1$ and $Z(f_W) =
W$ as a set (Theorem~\ref{thm:main} supplies such a function).  Define
\[
  \Phi_\mathcal{U} \colon \mRR \to \prod_{a \in Z} \C \Big/ \mathcal{U},
  \qquad \Phi_\mathcal{U}(h) = [(h(a))_{a \in Z}]_\mathcal{U}.
\]
Then:
\begin{enumerate}
\item[(i)] $\Phi_\mathcal{U}$ is a ring homomorphism to a field.  Here
  $\prod_{a \in Z} \C / \mathcal{U}$ is the ultraproduct of copies of
  $\C$ along $\mathcal{U}$; since $\mathcal{U}$ is a proper ultrafilter
  and each factor is a field, {\L}o\'s's theorem \cite{CK} implies the
  ultraproduct is a field.
\item[(ii)] $P_\mathcal{U} = \ker(\Phi_\mathcal{U}) \in \mathfrak{P}_1$:
  it is prime (kernel of a map to a field) and contains $f_Z$, which has
  constant term $1$.
\item[(iii)] $\mathcal{F}_Z(P_\mathcal{U}) = \mathcal{U}$: for $h = 1 +
  zH$ with $Z(h) \subseteq Z$, one has $h \in P_\mathcal{U}$ iff $\{a
  \in Z : h(a) = 0\} = Z(h) \in \mathcal{U}$.
\item[(iv)] The map $\mathcal{U} \mapsto P_\mathcal{U}$ is injective: if
  $P_{\mathcal{U}_1} = P_{\mathcal{U}_2}$ then $\mathcal{U}_1 =
  \mathcal{F}_Z(P_{\mathcal{U}_1}) = \mathcal{F}_Z(P_{\mathcal{U}_2}) =
  \mathcal{U}_2$.
\end{enumerate}
\end{proposition}

\begin{proof}
(i) The reduced product $\prod_{a \in Z}\C/\mathcal{U}$ is the
ultraproduct of copies of $\C$ along the proper ultrafilter
$\mathcal{U}$; by {\L}o\'s's theorem for the first-order theory of fields
it is a field, and $\Phi_\mathcal{U}$ is a ring homomorphism by
coordinatewise evaluation.  (ii) The kernel of a homomorphism to a field
is prime; $f_Z \in P_\mathcal{U}$ since $Z \in \mathcal{U}$ and $f_Z$
vanishes on $Z$.  (iii) Direct from the definition of the ultraproduct
and the restriction $Z(h) \subseteq Z$, which makes $\{a \in Z : h(a) =
0\}$ equal $Z(h)$.  (iv) Immediate from (iii).
\end{proof}

\begin{remark}
The map $\mathcal{U} \mapsto P_\mathcal{U}$ from ultrafilters on
admissible divisors into $\mathfrak{P}_1$ is injective by part~(iv).
Whether it is surjective---whether every element of $\mathfrak{P}_1$
arises as some $P_\mathcal{U}$---is open (Section~\ref{sec:open}).
\end{remark}

\section{Point-evaluation ideals}
\label{sec:pointeval}

For $a \in \D$ let $\mathrm{ev}_a \colon \mRR \to \C$, $f \mapsto f(a)$,
and $P_a = \ker(\mathrm{ev}_a)$.

\begin{proposition}[Primeness]
\label{prop:pointeval_prime}
For each $a \in \D$, the ideal $P_a$ is prime in $\mRR$.
\end{proposition}

\begin{proof}
Since every $f \in \mRR$ is holomorphic on $\D$ and $a \in \D$,
$\mathrm{ev}_a$ is a ring homomorphism, nonzero (the constant $1 \mapsto
1$) with codomain the integral domain $\C$; hence its kernel $P_a$ is
prime.
\end{proof}

\begin{proposition}[Maximality for real points]
\label{prop:pointeval_max_real}
For each $a \in \R \cap \D$, $P_a$ is a maximal ideal of $\mRR$ with
$\mRR/P_a \cong \R$.
\end{proposition}

\begin{proof}
Since all coefficients are real integers, $\mathrm{ev}_a(\mRR) \subseteq
\R$ for $a \in \R$.  The case $a = 0$ is excluded ($P_0 =
\mathfrak{n}_0$, with $\mRR/P_0 \cong \Z$ not a field; indeed $z \in
P_0$, so $P_0$ is type~(i)).  For $a \in (-1,1) \setminus \{0\}$ and any
$w \in \R$, define remainders $r_n = w - \sum_{k=0}^n c_k a^k$ by $r_{-1}
= w$ and, at step $n$, $c_n \in \Z$ the nearest integer to
$r_{n-1}/a^n$, so $r_n = r_{n-1} - c_n a^n$ and $\abs{r_n} \leq
\tfrac12\abs{a}^n$.

The bound $\abs{a} < 1$ keeps $\{c_n\}$ bounded: for $n \geq 1$,
$\abs{r_{n-1}} \leq \tfrac12\abs{a}^{n-1}$, so $\abs{r_{n-1}/a^n} \leq
\tfrac{1}{2\abs{a}}$ and $\abs{c_n} \leq \tfrac{1}{2\abs{a}} + \tfrac12$,
a constant independent of $n$.  Hence $\sup_n \abs{c_n} < \infty$, so
$\limsup \abs{c_n}^{1/n} = 1$ and $f = \sum_{n \geq 0} c_n z^n \in \mRR$.
Since $\abs{r_N} \leq \tfrac12\abs{a}^N \to 0$, the partial sums
$\sum_{k=0}^N c_k a^k = w - r_N \to w$, so $f(a) = w$.  Thus
$\mathrm{ev}_a \colon \mRR \twoheadrightarrow \R$ is surjective, $\mRR/P_a
\cong \R$ is a field, and $P_a$ is maximal.
\end{proof}

\begin{proposition}[Maximality in $\mR$]
\label{prop:pointeval_max_gaussian}
For each $a \in \D$, $P_a = \ker(\mathrm{ev}_a)$ is a maximal ideal of
$\mR = \Zi[[z]] \cap \OD$, with $\mR/P_a \cong \C$.
\end{proposition}

\begin{proof}
The same bounded radix expansion, with $c_n \in \Zi$ chosen nearest to
$r_{n-1}/a^n$, gives $\abs{r_n} \leq \tfrac{\sqrt2}{2}\abs{a}^n$ and
$\abs{c_n} \leq \tfrac{\sqrt2}{2\abs{a}} + \tfrac{\sqrt2}{2}$,
independent of $n$; so $f = \sum c_n z^n \in \mR$ with $f(a) = w$ for any
$w \in \C$.  Hence $\mR/P_a \cong \C$ and $P_a$ is maximal.
\end{proof}

\begin{remark}
Whether $P_a$ is maximal in $\mRR$ for $a \in \D \setminus \R$ remains
open: $\mathrm{ev}_a(\mRR)$ is a subring of $\C$ containing $\Z[a]$, but
surjectivity onto $\C$ with only integer (not Gaussian-integer)
coefficients is an arithmetic question about $a$.
\end{remark}

\section{A conditional B\'ezout property}
\label{sec:bezout}

\begin{theorem}[Conditional B\'ezout property for disjoint divisors]
\label{thm:bezout}
Let $f, g \in \mRR$ have disjoint zero sets in $\D$.  Then $(f,g)$ is not
contained in any type~(i) maximal ideal, nor in any point-evaluation
ideal $P_a$ with $a \in \R \cap \D$.  Consequently, if every type~(iii)
maximal ideal is a point-evaluation ideal $P_a$ for some $a \in \D$, then
$(f,g) = \mRR$.
\end{theorem}

\begin{proof}
Suppose $(f,g) \subseteq M$ for a maximal ideal $M$.  By
Proposition~\ref{prop:trichotomy}, $M$ is type~(i) or type~(iii).

\emph{Type~(i):} $z \in M$ forces $f(0) = g(0) = 0$ (as $M \cap \Z$
contains no unit), giving $0 \in Z(f) \cap Z(g)$, contradicting
disjointness.

\emph{Type~(iii), $M = P_a$ with $a \in \R \cap \D$:} then $f(a) = g(a)
= 0$, so $a \in Z(f) \cap Z(g)$, again contradicting disjointness.

If every type~(iii) maximal ideal is such a $P_a$, then $(f,g)$ lies in
no maximal ideal, hence $(f,g) = \mRR$.
\end{proof}

\begin{remark}
The conclusion is non-constructive even conditionally: Zorn's lemma
produces the maximal ideal and no explicit cofactors are exhibited.
Whether every type~(iii) maximal ideal is a point-evaluation ideal is the
key open question; see Section~\ref{sec:open}.
\end{remark}

\begin{remark}[Partial structure of $\mathrm{Spec}(\mRR)$]
\label{rem:spec}
Assembling the above: outside $\mathfrak{P}_1$ the primes are the $(p)$
(non-maximal) and the $(z,p)$ (maximal, type~(i)), by
Proposition~\ref{prop:primes_outside}.  Inside $\mathfrak{P}_1$ we have
the point-evaluation ideals $P_a$---maximal for $a \in \R \cap \D$
(Proposition~\ref{prop:pointeval_max_real}), prime for all $a \in \D$
(Proposition~\ref{prop:pointeval_prime})---and the kernels
$P_\mathcal{U}$ of ultraproduct maps, with $\mathcal{F}_Z(P_\mathcal{U})
= \mathcal{U}$ (Proposition~\ref{prop:ultraproductprime}).  Whether every
type~(iii) maximal ideal is a $P_a$, whether $P_a$ is maximal for
non-real $a$, and whether every element of $\mathfrak{P}_1$ is maximal
are open.
\end{remark}

\section{Further Questions}
\label{sec:open}

\begin{itemize}
\item \textbf{Classification of type~(iii) maximal ideals.}  The
  point-evaluation ideals $P_a$ ($a \in \D$) are type~(iii) maximal
  ideals.  Whether every type~(iii) maximal ideal arises this way ---
  equivalently, whether $\mathfrak{M} \cap \mathfrak{P}_1$ consists
  exactly of the $P_a$ --- is the central open question.  A
  positive answer yields the full B\'ezout property
  (Theorem~\ref{thm:bezout}) unconditionally.

\item \textbf{Surjectivity of the ultraproduct map.}  The injection
  $\mathcal{U} \mapsto P_\mathcal{U}$
  (Proposition~\ref{prop:ultraproductprime}) parametrizes a family of
  primes in $\mathfrak{P}_1$ by ultrafilters on admissible divisors.
  Whether it is surjective onto $\mathfrak{P}_1$, and whether the
  $P_\mathcal{U}$ are maximal, is open.

\item \textbf{Maximality inside $\mathfrak{P}_1$.}  Is every element of
  $\mathfrak{P}_1$ maximal?  Is $P_a$ maximal in $\mRR$ for non-real $a
  \in \D$?

\item \textbf{Surjectivity of $\phi$.}  Whether every maximal ideal of
  $\mR$ not containing $\mathfrak{n}_0$ arises by contraction from $\OD$
  is open; this would require $\mathfrak{n} \cdot \OD$ proper for every
  such maximal $\mathfrak{n} \subset \mR$.  (The contraction $\phi$ and
  its injectivity are from \cite{PaperI}.)

\item \textbf{Finer unit structure.}  Proposition~\ref{prop:units}
  characterizes units; the finer structure of $\mR^\times$ and
  $\mRR^\times$ is discussed in \cite{PaperI} and not revisited here.
\end{itemize}

\section*{Declaration on the Use of Artificial Intelligence}
In preparing this paper, the authors made use of AI language model
assistance (Claude) as a tool in the writing process---including
\LaTeX{} preparation and editing---and in exploratory work, such as
proof-checking and verifying arguments. All mathematical content and
results have been independently reviewed and verified by the authors,
who take full responsibility for the correctness, originality, and
presentation of this work. No AI system is an author of this paper,
nor was any AI system used to generate mathematical claims that were
not subsequently checked by the authors.

\end{document}